\documentclass[12pt]{amsart}
\usepackage{amsfonts,amssymb,amscd,amsmath,enumerate,verbatim,calc}

\newtheorem{Theorem}{Theorem}[section]
\newtheorem{Lemma}[Theorem]{Lemma}
\newtheorem{Corollary}[Theorem]{Corollary}
\newtheorem{Proposition}[Theorem]{Proposition}
\newtheorem{Remark}[Theorem]{Remark}

\newtheorem{Example}[Theorem]{Example}

\begin{document}
\title{Locating-Dominating sets in Hypergraphs}
\author{Muhammad Fazil, Imran Javaid$^{*}$, Muhammad Salman, Usman Ali}

\keywords{locating set, dominating set, locating-dominating set, hypergraph. \\
\indent 2010 {\it Mathematics Subject Classification.} 05C12, 05C65,
05C69.
\\
\indent $*$\ Corresponding author: ijavaidbzu@gmail.com}

\address{Centre for advanced studies in Pure and Applied Mathematics,
Bahauddin Zakariya University Multan, Pakistan. \newline E-mail:
mfazil@bzu.edu.pk, \{ijavaidbzu, solo33\}@gmail.com, uali@bzu.edu.pk}


\date{}
\maketitle
\begin{abstract}
A hypergraph is a generalization of a graph where edges can connect any number of
vertices. In this paper, we extend the study of locating-dominating sets to hypergraphs.
Along with some basic results, sharp bounds for the location-domination number of hypergraphs in
general and exact values with specified conditions are investigated. Moreover, locating-dominating sets in some specific hypergraphs are found.
\end{abstract}

\section{Introduction}
Let $G$ be a graph with vertex set $V(G)$ and edge set $E(G)$.
The number of elements in $V(G)$ and $E(G)$ is called the order and the
size of $G$, respectively. 
The {\it distance} between two vertices $u$ and $v$ in $G$, denoted
by $d(u,v)$, is the length of a shortest $u-v$ path in $G$. Let $u$
be a vertex of a graph $G$, then the \textit{open neighborhood} of
$u$ is $N(u)=\{v\in V(H)\ |\ uv\in E(G)\}$ and the \textit{closed
neighborhood} of $u$ is $N[u]=N(u)\cup \{u\}$. A set $D$ of
vertices of $G$ is a {\it dominating set} for $G$ if every vertex
$v$ in $V(G)- D$ has a neighbor in $D$, that is, for every
$v\in V(G)- D, N(v)\cap D\neq\emptyset$.

A set $\mathfrak{L}=\{x_1,\ldots,x_k\}$ of vertices of a graph $G$ is called a
{\it locating set} if for every two distinct vertices $u$ and $v$ of $G$, $(d(u,x_1),\ldots,d(u,x_k))\neq (d(v,x_1),\ldots,d(v,x_k))$. The
\emph{location number} (also called the {\it metric dimension} \cite{mel}) is
the minimum cardinality of a locating set of $G$ \cite{slater4}.

A set $S$ of vertices of a graph $G$ is called a
{\it locating-dominating set} if it is both the locating and
dominating set. An elaborate and more general definition is the
following: A set $S$ of vertices of $G$ is called a
locating-dominating set for $G$ if for every two distinct elements $u,v \in V(G)- S$, we have $\emptyset \neq N(u)\cap S\neq
N(v)\cap S\neq\emptyset$. The {\it location-domination number},
denoted by $\lambda(G)$, is the minimum cardinality of a locating-dominating
set of $G$ \cite{slater2}.


Locating-dominating sets in graphs were firstly studied by Slater
\cite{slater2}. The motivations of locating-dominating sets comes, for
instance, from fault diagnosis in multiprocessor systems. Such a
system can be modeled as a graph where vertices are processors and
edges are links between processors. A considerable literature has
been developed in this field (see \cite{ber,char1,col,fin,hon,rall,slater1,slater3}). The
decision problem for locating-dominating sets for directed graphs
has been shown to be an NP-complete problem \cite{char2}. In
\cite{cac}, it was pointed out that each locating-dominating set is
both the locating and dominating set. However, a set that is both the locating
and dominating is not necessarily a locating-dominating set.

A {\it hypergraph} $H$ is a pair $(V(H),E(H))$, where $V(H)$ is a
finite set of vertices and $E(H)$ is a finite family of non-empty
subsets of $V(H)$, called hyperedges, with $\bigcup\limits_{E\in
E(H)} E = V(H)$. A \emph{subhypergraph} $K$ of a hypergraph $H$ is a
hypergraph with vertex set $V(K)\subseteq V(H)$ and edge set
$E(K)\subseteq E(H)$. The {\it rank} of $H$, denoted by $rank(H)$,
is the maximum cardinality of a hyperedge in $E(H)$. $H$ is {\it
linear} if for distinct hyperedges $E_i$ and $E_j$, $|E_i\cap E_j|
\leq 1$, so in a linear hypergraph, there may be no repeated
hyperedges of cardinality greater than one. A hypergraph $H$ with no
hyperedge is a subset of any other hyperedge is called {\it Sperner}.

A vertex $v\in V(H)$ is said to be {\it incident} with a hyperedge $E$ of $H$
if $v\in E$. If $v$ is incident with exactly $n$ hyperedges, then we
say that the {\it degree} of $v$ is $n$; if all the vertices $v \in
V(H)$ have degree $n$, then $H$ is {\it n-regular}. The maximum
degree of any vertex in $H$ is denoted by $\Delta(H)$. Similarly, if
there are exactly $n$ vertices incident with a hyperedge $E$, then
we say that the size of $E$ is $n$; if all the hyperedges $E \in
E(H)$ have size $n$, then $H$ is {\it n-uniform}. A simple graph is
a $2$-uniform hypergraph.

A {\it path} from a vertex $v$ to another vertex $u$, in a
hypergraph, is a finite sequence of the form
$v,E_1,w_1,E_2,w_2,\ldots,E_{l-1}, w_{l-1},E_l,u$, having {\em length}
$l$ such that $v\in E_1,\; w_i\in E_i\cap E_{i+1}\; \mbox{for}\;
i=1,2,\ldots,{l-1}$ and $u\in E_l.$ A hypergraph $H$ is {\it
connected} if there is a path between every two vertices of $H$. {\it All the
hypergraphs considered in this paper are connected Sperner
hypergraphs}.

A hypergraph $H$ is said to be a {\it hyperstar} if $E_i\cap E_j =
\mathcal{C} \neq \emptyset$, for any $E_i,E_j \in E(H)$. We will
call $\mathcal{C}$, the {\it center} of the hyperstar. If there
exists a sequence of hyperedges $E_1,E_2,\ldots,E_k$ in a hypergraph
$H$, then $H$ is said to be (1)\ a {\it hyperpath} if $E_i\cap E_j
\neq \emptyset$ if and only if $|i-j| = 1$; (2)\ a {\it hypercycle}
if $E_i\cap E_j \neq \emptyset$ if and only if $|i-j| = 1$ (mod
$k$). A connected hypergraph $H$ with no hypercycle is called a {\it
hypertree}.

In graphs, the theory of dominating sets and locating-dominating sets is extensively studied. Hypergraphs, in the context of domination, were firstly considered by Behr and Camarinopoulos in 1998 \cite{behr}, and further considered by Acharya \cite{achar1,achar2} and Jose and Tuza \cite{jose}. In this paper, we consider hypergraphs in the context of location-domination. We give some sharp lower bounds for the location-domination
number of hypergraphs. Also, we investigate the location-domination number of
some well-known families of hypergraphs such as hyperpaths,
hypercycles and $k$-partite hypergraphs.

\section{Some Basic Results and Bounds}
Two vertices $u$ and $v$ of a hypergraph incident with the same hyperedge are said to be {\it coincident} vertices. Let
$E_{i}^{(d)}=\{E_{i_{1}}, E_{i_{2}},\ldots, E_{i_{d}}\}$ be a
collection of hyperedges. We denote the set of all
the vertices, having degree $d$, incident with every hyperedge in $E_{i}^{(d)}$ by $S_{i}^{(d)}$, and we call it
the $i$th {\it coincident set} of the vertices having
degree $d$. It should be noted that for each set
$E_{i}^{(d)}\subseteq E(H)$, there corresponds a coincident set $S^{(d)}_i$, which may be empty.

Coincident vertices have the same degree but two vertices having same
degree may not be coincident as in the following example
(illustrating the notion of the coincident set), the vertices $v_{6}$ and $v_{8}$
are not coincident although the degree of both the vertices is same.

\begin{Example}\label{exp2.1}Let $V(H)=\{v_{1},v_{2},\ldots,v_{10}\}$ and $E(H)=\{E_{1},
E_{2},E_{3}\}$, where $E_{1}=\{v_{1},v_{2},v_{3},v_{4}\}$,
$E_{2}=\{v_{1},v_{4},v_{5},v_{6},v_{7}\}$ and $E_{3}=\{v_{2},
v_{3},v_{4},v_{5},v_{8},v_{9},v_{10}\}$. Then we can write,
$E_{1}^{(1)}=\{E_{1}\},E_{2}^{(1)}=\{E_{2}\},E_{3}^{(1)}=\{E_{3}\},E_{1}^{(2)}=\{E_{1},E_{2}\},E_{2}^{(2)}=\{E_{1},E_{3}\}$,
$E_{3}^{(2)}=\{E_{2}, E_{3}\}, E_{1}^{(3)}=\{E_{1},E_{2},E_{3}\}$
and the corresponding coincident sets are:
$S_{1}^{(1)}=\emptyset, S_{2}^{(1)}=\{v_{6},v_{7}\},
S_{3}^{(1)}=\{v_{8},v_{9},v_{10}\} , S_{1}^{(2)}=\{v_{1}\},
S_{2}^{(2)}=\{v_{2},v_{3}\}, S_{3}^{(2)}=\{v_{5}\}$ and
$S_{1}^{(3)}=\{v_{4}\}$.
\end{Example}

\begin{Remark}\label{rem2.2} Two vertices belonging to different coincident sets can have the same closed
neighborhood. As in Example \ref{exp2.1}, the vertices $v_{4}$ and $v_{5}$ belong to
different coincident sets, namely $S_{1}^{(3)}$ and $S_{3}^{(2)}$,
respectively. However, $N[v_{4}]=N[v_{5}]=V(H)$.
\end{Remark}

Now, we discuss some properties of coincident sets in the following
proposition:

\begin{Proposition}\label{prop2.3}
(1) The set of all non-empty coincident sets in a hypergraph $H$ partitions $V(H)$.\\
(2) The number of non-empty coincident
sets in a hypergraph $H$ is bounded above by $\sum \limits_{d =
1}^{\Delta(H)} {m\choose d}$, where $m$ is the size of $H$.
\end{Proposition}

\begin{proof}$(1)$\ Note that $\cup S_{i}^{(d)}\subseteq V(H)$. Let $v\in V(H)$, then
$v\in E_{i_{l}}$ for some hyperedge $E_{i_{l}}\in E_{i}^{(d)}$. Since for each set $E_i^{(d)}$ of hyperedges, there
corresponds a coincident set $S_{i}^{(d)}$ yielding $v\in
S_{i}^{(d)}$ for some $i$. Hence, $V(H)\subseteq\cup
S_{i}^{(d)}$ implies $\cup S_{i}^{(d)}=V(H)$. Now, we show that $S_i^{(d)}\cap
S_{j}^{(d)}=\emptyset$ , or $S_i^{(d)}=
S_{j}^{(d)}$ for $i\neq j$.

Suppose that $S_i^{(d)}\cap S_{j}^{(d)}\neq \emptyset$ and let $v\in S_{i}^{^{(d)}}\cap S_{j}^{^{(d)}}\ (i\neq j)$. Then by
definition, $v$ will be incident with each hyperedge $E_{i_{l_{1}}}$ of
$E_{i}^{(d)}$ and $E_{j_{l_{2}}}$ of $E_{j}^{(d)}$, respectively.
This implies that $v\in E_{i_{l_{1}}}\cap E_{j_{l_{2}}}$. But, the
considered graph is connected Sperner hypergraph so $v\in
E_{i_{l_{1}}}\cap E_{j_{l_{2}}}$ holds only if $E_{i_{l_{1}}}=
E_{j_{l_{2}}}$, which shows that
$S_{i}^{^{(d)}}=S_{j}^{^{(d)}}$. It concludes the required result.

\noindent $(2)$\ By the definition of coincident set, for each
$E_{i}^{(d)}\subseteq E(H)$, there corresponds a coincident set
$S_{i}^{(d)}$. Also, since there are ${m\choose d}$ possible
subsets of $E(H)$ of cardinality $d$, so there will be at most as
many coincident sets for each $d$.
\end{proof}


If $N[u]=N[v]$ for any two vertices $u$ and $v$ of a hypergraph $H$, then for each
$S\subseteq V(H)$, $N(u)\cap S=N(v)\cap S$, which implies that either
$u$ or $v$ should belong to every locating-dominating set of $H$. Hence,
we have the following straightforward result:

\begin{Lemma}\label{lem2.4}
Let $S$ be a locating-dominating set for a hypergraph $H$. If for
any two distinct vertices $u$ and $v$ of $H$, $N[u]=N[v]$, then at
least one of $u$ and $v$ must belong to $S$.
\end{Lemma}

By the definition of coincident set $S_{i}^{(d)}$, for every two
distinct elements $u,v$ of $S_{i}^{(d)}$, $N[u]=N[v]$. Hence, by
Lemma \ref{lem2.4}, we have the following result:

\begin{Lemma}\label{lem2.5}
For any locating-dominating set $S$ of a hypergraph and for a non-empty coincident set
$S_{i}^{(d)}$, we have $|S|\geq |S\cap S_{i}^{(d)}|\geq|S_{i}^{(d)}|-1$.
\end{Lemma}


For a locating-dominating set $S$ of a hypergraphs $H$ and for a coincident set
$S_{i}^{(d)}$, we let $C_{i}^{(d)}=S\cap S_{i}^{(d)}$ whenever
$S_{i}^{(d)}\neq \emptyset$, and $C_{i}^{(d)}=\emptyset$ whenever
$S_{i}^{(d)}= \emptyset$. Moreover, we let
$C=\bigcup\limits_{d=1}^{m}\bigcup\limits_{i=1}^{m\choose
d}{C_{i}^{(d)}}$. Since $S_{i}^{(d)}=\emptyset$ when $d>\Delta(H)$,
therefore
$C=\bigcup\limits_{d=1}^{\Delta(H)}\bigcup\limits_{i=1}^{m\choose
d}{C_{i}^{(d)}}$.

A lower bound for the location-domination
number of a hypergraph $H$ is given in the following
result:

\begin{Theorem}\label{th2.6}
Let $S$ be a minimum locating-dominating set for a hypergraph
$H$ with $m\geq 2$ hyperedges. Then
$$ |S| \geq |C|=\sum \limits_{d=1}^{\Delta(H)}\sum \limits_{i=1}^{m\choose
d}{|C_{i}^{(d)}|}.$$
\end{Theorem}

\begin{proof}
Since non-empty coincident sets in a hypergraph $H$ form a partition of $V(H)$, by Proposition \ref{prop2.3}. So, every two coincident sets are either disjoint or equal, which yields that
$$|C|=\sum
\limits_{d=1}^{\Delta(H)}\sum \limits_{i=1}^{m\choose
d}{|C_{i}^{(d)}|}.$$
Further, Lemma \ref{lem2.5} straightforwardly concludes that $|S|\geq |C|$.
\end{proof}

The following result describes that the lower bound established above
is sharp.

\begin{Theorem}\label{th2.7}
Let $H$ be a hypergraph in which every hyperedge
contains at least two vertices of degree one. If $S$ is a minimum locating-dominating set for $H$, then
$$|S|=|C|=\sum
\limits_{d=1}^{\Delta(H)}\sum \limits_{i=1}^{m\choose
d}{|C_{i}^{(d)}|}.$$
\end{Theorem}

\begin{proof} By hypothesis, $C_{i}^{(1)}\neq \emptyset$ for each
$i=1,2,\ldots,m$. 
Let $u,v\in V(H)- C$. Since $|S_{i}^{(d)}|-|C_{i}^{(d)}|\leq
1$, there exist distinct coincident sets $S_{i}^{(d_1)}$ and
$S_{j}^{(d_2)}$ such that $u\in S_{i}^{(d_1)}$ and $v\in
S_{j}^{(d_2)}$. Assuming $d_1\geq d_2$, there exists a hyperedge $E_k$
such that $u\in E_k$ and $v\notin E_k $. Therefore $N(u)\cap C \neq
N(v)\cap C$, and hence, together with Theorem \ref{th2.6}, we have the required result.
\end{proof}

We give two examples which show that the condition in Theorem
\ref{th2.7} cannot be relaxed generally.

\begin{Example}\label{exp2.8}
Let $H$ be a hypergraph with vertex set
$V(H)=\{v_1,v_2,v_3,v_4,v_5\}$ and edge set
$E(H)=\left\{E_1=\{v_1,v_2,v_3,v_4\},E_2=\{v_4,v_5\}\right\}.$
Clearly, $C_{2}^{(1)}=\emptyset$ and the
condition of Theorem \ref{th2.7} is not satisfied. Observe that $C=\{v_1,v_2\}$. But, $C$ is
not a locating-dominating set for $H$ because $N(v_3)\cap C =
N(v_4)\cap C$.
\end{Example}

\begin{Example}\label{exp2.9}
Let $H$ be a hypergraph with vertex set
$V(H)=\{v_1,v_2,v_3,v_4,v_5,v_6\}$ and edge set
$E(H)=\left\{E_1=\{v_1,v_2,v_3,v_4\},E_2=\{v_3,v_4,v_5,v_6\},
E_3=\{v_1,v_2,v_5,v_6\}\right\}.$ Clearly, $ C_{i}^{(1)}=\emptyset$
but $C_{i}^{(2)}\neq \emptyset$, for all $i = 1,2,3$. Observe that
$C=\{v_1,v_3,v_5\}$. But, $C$ is not a locating-dominating set for
$H$ because $N(v_2)\cap C = N(v_4)\cap C$.
\end{Example}

A hypergraph $H$ is said to be a {\it complete hypergraph} if for all
$\{u,v\}\subseteq V(H)$, there is $E_{i}\in E(H)$ such that
$\{u,v\}\subseteq E_{i}$. A {\it clique} of a hypergraph $H$, denoted by
$\widetilde{H}=(V(\widetilde{H}),E(\widetilde{H}))$, is a complete
subhypergraph of $H$.

A sharp upper bound for the location-domination number
of a hypergraph $H$ is given in the following lemma:

\begin{Lemma}\label{lem2.18}
If $S$ is a locating-dominating set for a hypergraph
$H$ with $n$ vertices, then $\lambda(H)\leq |S| \leq n-1$, and this bound is
sharp.
\end{Lemma}

\begin{proof}
It is easy to see that any $n-1$ vertices of $H$ form a locating-dominating set $S$ for $H$, which implies $\lambda(H)\leq |S| \leq n-1$.

For sharpness, consider a complete hypergraph $H$ of order $n$. Since $N[u]=N[v]$ for every two distinct vertices $u$ and $v$ of $H$, so it never be hold that a set $S$ with $|S|< n-1$ forms a locating-dominating set for $H$. For otherwise, there exist $x,y\in V(H)- S$ such that $N(x)\cap S = N(y)\cap S$.
\end{proof}

\begin{Lemma}\label{lem2.19}
Let $\widetilde{H}$ be a clique of hypergraph $H$ and $S$ be a locating-dominating set for $H$.
If $S\subseteq V(\widetilde{H})$, then $|S|\geq |V(\widetilde{H})|-1$.
\end{Lemma}

\begin{proof} Since all the vertices of $\widetilde{H}$ are mutually
coincident, therefore they have the same closed neighborhoods. Hence, by
Lemma \ref{lem2.4}, at least $|V(\widetilde{H})|-1$ elements from
$V(\widetilde{H})$ contained in $S$.
\end{proof}

A {\it vertex packing} in a hypergraph $H$ is a
subset $\mathcal{P}\subseteq V(H)$ such that no two elements of
$\mathcal{P}$ belong to the same hyperedge of $H$. The {\it packing
number} is the maximum cardinality of such a set $\mathcal{P}$, and we denote it by $\pi$ \cite{schei}.

\begin{Theorem}\label{th2.24}
Let $H$ be a linear hypergraph of order $n$ with $S_{i}^{(1)}\geq 2$ for all $i$. Then $\lambda(H)\leq n-\pi$ and this bound is sharp.
\end{Theorem}

\begin{proof}
By the definition, for any two distinct $u,v\in\mathcal{P}$, we have $N(u)\neq N(v)$. So,
$V(H)-\mathcal{P}$ is a locating-dominating set for $H$. Hence
$\lambda(H)\leq n-\pi$ since $\pi$ is the largest size of packing $\mathcal{P}$. Further, the bound is sharp if $H$ is a complete hypergraph.
\end{proof}

\section{Location-Domination in Some Specific Hypergraphs}
Since in a uniform linear hypergraph, we have
$C_{i}^{(2)}= \emptyset$ and $C_{i}^{(1)}\neq \emptyset$ for all $i$. So, Theorem \ref{th2.7} yields the following consequence:

\begin{Corollary}\label{cor2.10}
Let $H$ be a $k$-uniform linear hypergraph with $m$ hyperedges.
Then for $k\geq 4$, $\lambda(H) =
\sum\limits_{i=1}^{m}|C^{(1)}_{i}|$.
\end{Corollary}

\begin{Corollary}\label{cor2.11}
Let $H$ be a $k$-uniform linear hyperpath with
$m$ hyperedges. Then for $k\geq 4$, $\lambda(H)=m(k-3)+2$.
\end{Corollary}

\begin{proof} Observe that $|C_{1}^{(1)}|=|C_{m}^{(1)}|=k-2$,
whereas for $2\leq i\leq m-1$, $|C_{i}^{(1)}|=k-3$ hence, we have
$\lambda(H)=m(k-3)+2$.
\end{proof}

\begin{Corollary}\label{cor2.12}
Let $H$ be a $k$-uniform linear hypercycle with $m$
hyperedges. Then for $k\geq 4$, $\lambda(H)=m(k-3)$.
\end{Corollary}

\begin{proof}
Note that, for every $k$-uniform linear hypercycle
$|C_{i}^{(1)}|=k-3$, for all $1\leq i\leq m$. Therefore
$\lambda(H)=m(k-3)$.
\end{proof}

\begin{Theorem}\label{th2.13}
In a $k$-uniform hypercycle $C_{m,k}\ (k\geq 4)$ with
$m\geq 3$ hyperedges, if $C_{i}^{(1)}=\emptyset$ for all
$i$ and the order of each non-empty coincident set $S_{i}^{(2)}$ is same, then
$\lambda(C_{m,k})=m(\lfloor\frac{k}{2}\rfloor-1).$
\end{Theorem}

\begin{proof} Since $C_{i}^{(1)}=\emptyset$ for all $i$ and each non-empty coincident set
$S_{i}^{(2)}$ has the same order, so each $C_{i}^{(2)}$ has exactly
$\lfloor\frac{k}{2}\rfloor-1$ elements. Since there are $m$ hyperedges, therefore $\lambda(C_{m,k})=m(\lfloor\frac{k}{2}\rfloor-1).$
\end{proof}

The following result for hyperpaths also shows that the lower bound established in Theorem \ref{th2.6}
is sharp.

\begin{Theorem}\label{th2.14}
Let $H$ be a hyperpath with $m$ hyperedges. If $S_{i}^{(1)}$, for $i=1,m$, and each non-empty coincident set
$S_{i}^{(2)}$ has at least two elements, then
$$\lambda(H)=\sum
\limits_{d=1}^{\Delta(H)}\sum \limits_{i=1}^{m\choose
d}{|C_{i}^{(d)}|}.$$
\end{Theorem}

\begin{proof}
If $C_{i}^{(1)}\neq \emptyset$, for all $i$. Then, by Theorem
\ref{th2.7}, we have the required result. If
$C_{i}^{(1)}=\emptyset$ for $2\leq i\leq m-1$, then we have the following
three cases for $S_{i}^{(1)}$.

\noindent \textbf{Case 1:}\ ($S_{i}^{(1)} \neq
\emptyset$ for all $2\leq i\leq m-1$).\\
Let $v,v'\in V(H)- C$. If $v\in E_1$ and $v'\in E_m$. Then, clearly,
$N(v)\cap C \neq N(v')\cap C$ because $(N(v)\cap C)\subset E_1$ and
$(N(v')\cap C)\subset E_m$.

If $v$ and $v'$ are common vertices, then again $N(v)\cap C \neq
N(v')\cap C$ because $H$ is Sperner and each $S_{i}^{(2)}$ is
disjoint.

If $v\in E_i$ and $v'\in E_j$, for $2\leq i,j\leq m-1$ and $i\neq j$. Then $N(v)=\{E_{i-1}\cap E_i\}\cup \{E_i\cap E_{i+1}\} \neq
\{E_{j-1}\cap E_j\} \cup \{E_j\cap E_{j+1}\}=N(v')$, and hence
$N(v)\cap C\neq N(v')\cap C$.

If $v\in E_{i}$ and $v'$ is a common vertex. Then by above
discussion, we note that both the vertices have their distinct open
neighborhoods in $C$.

The other two cases when $S_{i}^{(1)} \neq \emptyset$ for some
$2\leq i\leq m-1$ and when $S_{i}^{(1)}= \emptyset$ for all $2\leq i\leq
m-1$ follows from Case 1.
\end{proof}

We know that a 2-uniform linear hyperpath (hypercycle) is a simple path (cycle). The exact value for $\lambda(H)$ of a simple path (cycle) is
already determined by Slater.
\begin{Theorem}\label{th2.32}\cite{slater2}
Let $P_{n\geq 2}$ be a simple path and $C_{n\geq 3}$ be a simple cycle.
Then $\lambda(P_{n})= \lambda(C_{n})=\lceil\frac{2n}{5}\rceil$.
\end{Theorem}

Observe that, in the case of 3-uniform linear hyperpath $H$ with two
hyperedges, $\lambda(H)=2$. We also observe that, the 3-uniform
linear hyperpath with three and four hyperedges, respectively, has the location-domination number $3$ and $4$, respectively. In the next result, we determine the explicit value for the location-domination number for 3-uniform linear hyperpaths with more than four
hyperedges.

\begin{Theorem}\label{th2.33}
Let $P_{m,3}$ be a 3-uniform linear hyperpath with
$m\geq 5$ hyperedges. Let $m=3a+b+2$, where $a\geq 1$ and $0\leq b\leq 2$.
Then $\lambda (P_{m,3})=2a+b+2$.
\end{Theorem}

\begin{proof}Let $v_{i}$ represents a vertex of degree one in the
hyperedge $E_{i}$ and $v_{i,i+1}\in E_{i}\cap E_{i+1}$ is common vertex
of degree two. Since $E_{1}$ and  $E_{m}$ contains two vertices of
degree one, we denote them as $v_{1}$, $v_{1}'$ for $E_{1}$ and $v_{m}$, $v_{m}'$ for $E_{m}$.

Our claim is that $\lambda(P_{m,3})=2a+b+2$. We prove it by showing that a set $S\subseteq
V(H)$ with $|S|=2a+b+2$ is a minimum locating-dominating set for $P_{m,3}$.
We consider the following three cases:

\noindent \textbf{Case 1:}\ $(b=0)$\ Let $S=\{v_{1}',v_{m}'\}\cup
\{v_{3i+2,3(i+1)},v_{3(i+1),3(i+1)+1}\ |\ 0\leq i\leq a-1\}$. This set $S$ is a minimum locating-dominating set of order $2a+2$ because of the following unequal $S$-neighborhoods: $N(v_{1})\cap S=\{v_{1}'\}$; $N(v_{m})\cap S=\{v_{m}'\}$; $N(v_{3i+1, 3i+2})\cap
S=\{v_{3i,3i+1}, v_{3i+2,3(i+1)}\}$ with $v_{0,1}=v_{1}'$;
$N(v_{3i+2})\cap S=\{v_{3i+2,3(i+1)}\}$; $N(v_{3(i+1)})\cap
S=\{v_{3i+2,3(i+1)},v_{3(i+1),3(i+1)+1}\}$; $N(v_{3(i+1)+1})\cap
S=\{v_{3(i+1),3(i+1)+1}\}$ and $N(v_{m-1,m})\cap
S=\{v_{m-2,m-1},$ $v_{m}'\}$.

\noindent \textbf{Case 2:}\ $(b=1)$\ The set $S=\{v_{1}',v_{m}'\}\cup \{v_{3i+2,3(i+1)},v_{3(i+1),3(i+1)+1}\ |\ 0\leq
i\leq a-1\}\cup \{v_{m-1,m}\}$ is a required minimum locating-dominating
set of order $2a+3$ because of the following unequal $S$-neighborhoods of all the elements in $V(P_{m,3})- S$: $N(v_{m})\cap S=\{v_{m-1,m}, v_{m}'\}$;
$N(v_{m-2,m-1})\cap S=\{v_{m-3,m-2},v_{m-1,m}\}$; $N(v_{m-1})\cap
S=\{v_{m-1,m}\}$ and those listed in Case 1.

\noindent \textbf{Case 3:}\ $(b=2)$\ Let $S=\{v_{1}',v_{m}'\}\cup \{v_{3i+2,3(i+1)},v_{3(i+1),3(i+1)+1}\ |\ 0\leq
i\leq a-1\}\cup \{v_{m-2,m-1}, v_{m-1,m}\}$. Together with the unequal $S$-neighborhoods of all the
elements in $V(P_{m,3})- S$ listed in the Cases 1 and 2, and the $S$-neighborhood $N(v_{m-1})\cap S=\{v_{m-2,m-1}, v_{m-1,m}\}$, one can conclude that $S$ is a minimum locating-dominating set of order $2a+4$.
\end{proof}

In the next result, we determine the explicit value for the location-domination number for 3-uniform linear hypercycles having more than five
hyperedges.

\begin{Theorem}\label{th2.34}
Let $C_{m,3}$ be a 3-uniform linear hypercycle with
$m\geq 6$ hyperedges. Let $m=3a+b$, where $a\geq 2$ and $0\leq b\leq 2$.
Then $\lambda(C_{m,3})=2a+b$.
\end{Theorem}

\begin{proof} In $C_{m,3}$, each $v_{i}\in E_{i}$ represents a vertex
of degree one and $v_{i,i+1}\in E_{i}\cap E_{i+1}$ with
$v_{m,m+1}=v_{m,1}$. We prove that $\lambda (C_{m,3})=2a+b$ by
showing that a set $S\subseteq V(H)$ with $|S|=2a+b$ is a minimum
locating-dominating set for $C_{m,3}$. We discuss the following three cases for $b$:

\noindent \textbf{Case 1:}\ $(b=0)$\ Let $S= \{v_{3i+1,3i+2},v_{3i+2,3i+3}\ |\ 0\leq i\leq a-1\}$. Then
$|S|=2a$. One can see that $S$ is a minimum locating-dominating set because of the following unequal $S$-neighborhoods: $N(v_{3i+1})\cap
S=\{v_{3i+1,3i+2}\}$; $N(v_{3i+2})\cap
S=\{v_{3i+1,3i+2},v_{3i+2,3i+3}\}$; $N(v_{3i+3})\cap
S=\{v_{3i+2,3i+3}\}$; $N(v_{3i+3,3(i+1)+1})\cap
S=\{v_{3i+2,3i+3},\\ v_{3(i+1)+1,3(i+1)+2}\}$.

\noindent \textbf{Case 2:}\ $(b=1)$\ The set $S= \{v_{3i+1,3i+2},v_{3i+2,3i+3}\ |\ 0\leq i\leq a-1\}\cup
\{v_{m,1}\}$ is a minimum locating-dominating set of order $2a+1$ because all the elements in $V(C_{m,3})- S$ have unequal $S$-neighborhoods
$N(v_{1})\cap S=\{v_{m,1}, v_{1,2}\}$; $N(v_{m})\cap
S=\{v_{m,1}\}$ and those listed in Case 1.

\noindent \textbf{Case 3:}\ $(b=2)$\ Let $S= \{v_{3i+1,3i+2},v_{3i+2,3i+3}\ |\ 0\leq i\leq a-1\}\cup
\{v_{m-1,m},v_{m,1}\}$ be the set of order $2a+2$. Then the $S$-neighborhoods $N(v_{1})\cap
S=\{v_{m,1},v_{1,2}\}$; $N(v_{m-1})\cap S=\{v_{m-1},v_{m}\}$;
$N(v_{m})\cap S=\{v_{m-1,m},v_{m,1}\}$ and those listed in Case 1 all are unequal, which implies that $S$ is a minimum locating-dominating set.
\end{proof}

\begin{Proposition}\label{prop2.15}
Let $H$ be a hyperstar with $n$ vertices and $m\geq 2$
hyperedges. If for all $1\leq i\leq m$, $|S_{i}^{(1)}|= 1$ and
$|\mathcal{C}|>1$, then $\lambda(H)= n-2$, where $\mathcal{C}$ is the center of the hyperstar.
\end{Proposition}

\begin{proof} Suppose $A=\mathcal{C}$ and $B=\bigcup \limits
_{i=1}^{m}S_{i}^{(1)}$. Since any two vertices of $A$ have their
same closed neighborhood so, by Lemma \ref{lem2.4}, $|A|-1$ vertices
of $A$ must belong to a minimum locating-dominating set $S$ for $H$. Further, for any two elements $u_{1}$ and $u_{2}$
of $B$, we have $N(u_{1})=N(u_{2})$. It follows that if there exists
$\{u_{1},u_{2}\}\subseteq B- S$, then $N(u_{1})\cap
S=N(u_{2})\cap S$, which implies that $S$ contains exactly $|B|-1$
vertices of $B$. Since $A\cup B= V(H)$. Hence, $\lambda(H)= |A|+|B|-2$.
\end{proof}

If each coincident set $S_i^{(1)}\ (1\leq i\leq m)$ in a hyperstar with $m$ hyperedges has at least two elements, then we have the following straightforward proposition:

\begin{Proposition}\label{prop2.16} Let $H$ be a hyperstar with $m \geq 2$
hyperedges and for all $i$, $|S_{i}^{(1)}|> 1$, then $\lambda(H)=
\sum\limits_{i=1}^{m}|C^{(1)}_{i}|+|\mathcal{C}|-1$, where $\mathcal{C}$ is the center of the hyperstar.
\end{Proposition}

In a $k$-uniform linear hyperstar,
$|C_{i}^{(1)}|=k-2$, for all $1\leq i\leq m$, therefore we have the following consequence:

\begin{Corollary}\label{cor2.17}
Let $H$ be a $k$-uniform linear hyperstar with $m \geq 3$
hyperedges. Then for $k\geq 3$, $\lambda(H)= m(k-2)$.
\end{Corollary}

Let $N_{1},N_{2},\ldots,N_{t}$ be $t$ disjoint finite sets with
$|N_{i}|=n_{i}$. A {\it complete $t$-partite $r$-uniform hypergraph}
is $H = K_{n_{1},n_{2},\ldots,n_{t}}^{r}$ has the vertex set $V(H) = \bigcup \limits_i^t N_{i}$ and for each subset $E_j\subseteq V(H)$, $E_j\in E(H)$ if $|E_j|=r$
and $|E_j\cap N_{i}|\leq 1$ \cite{bergc}.

\begin{Lemma}\label{lem2.21}
For each $t\in \mathbb{Z^{+}}$ and $t\geq 2$, if $S$ is a locating-dominating set for $K_{n_{1},n_{2},\ldots,n_{t}}^{r}$, then
$|S\cap N_{i}|\geq n_{i}-1$ for all $1\leq i\leq t$.
\end{Lemma}

\begin{proof}
Suppose that $A=N_{t_{1}},1\leq t_{1}\leq t$. and $B=\bigcup \limits_{\substack{
i=1\\
i\neq t_1}}^t N_{i}$. Then for each $u\in A$ and for each $v\in B$, we have
$N(u)\neq N(v)$ whereas, for any two elements $u_{1},u_{2}$ of $A$,
we have $N(u_{1})= N(u_{2})$. Hence, if there exists
$\{u_{1},u_{2}\}\subseteq A- S$, then we have $N(u_{1})\cap
S= N(u_{2})\cap S$, which implies $|S\cap N_{i}|\geq n_{i}-1$.
\end{proof}

\begin{Theorem}\label{th2.22}
Let $K_{n_{1},n_{2},\ldots,n_{t}}^{r}$ be a complete t-partite r-uniform hypergraph.
Then for $t\in \mathbb{Z^{+}}$ and for all $i$, $v_{i}\in
N_{i}$, $S= \bigcup\limits_i^t (N_{i}- \{v_{i}\})$ is a locating-dominating
set for $K_{n_{1},n_{2},\ldots,n_{t}}^{r}$ if and only if there is at most
one partite set of cardinality 1.
\end{Theorem}

\begin{proof}
Suppose that there are two partite sets $N_{i_{t_{1}}}$ and $N_{i_{t_{2}}}$ such
that $n_{i_{t_{1}}}=n_{i_{t_{2}}}=1$. Let $u\in N_{i_{t_{1}}}$ and
$v\in N_{i_{t_{2}}}$. Then, by the definition of
$K_{n_{1},n_{2},\ldots,n_{t}}^{r}$, both the
vertices have their same open neighborhoods in $S$, which implies that $S$ is not a locating-dominating set, a contradiction.

Conversely, let $R=\bigcup \limits _{i=1}^{t}(N_{i}- S)$
and $N_{i_{t}}$ is the unique partite set such that $n_{i_{t}}=1$. Then, by
the definition of $S$, the set $R$ has exactly $t$ vertices of
$K_{n_{1},n_{2},\ldots,n_{t}}^{r}$, that is, $R$ has exactly one vertex from each
$N_{i}$. Since every two partite sets are disjoint, so every vertex in $R$ has
its different open neighborhood in $S$.
\end{proof}

\begin{Remark}\label{rem2.23}
In $K_{n_{1},n_{2},\ldots,n_{t}}^{r}$, note that, if there are $p$ partite
sets of cardinality 1, then all the vertices in these partite sets
except one will belong to every locating-dominating set for
$K_{n_{1},n_{2},\ldots,n_{t}}^{r}$.
\end{Remark}

Let $H$ be a hypergraph and let $k\in \mathbb{Z}^{+}$. The
{\it $k$-section} of $H$ is $H_{k}=(V(H),E_{k})$, where for a set
$X\subseteq V(H)$, $X$ belongs to $E_{k}$ if any of the following conditions
holds:\\
(1)\ $|X|\leq k$ and $X\in E(H)$,\\
(2)\ $|X|=k$ and there exists $E_{j}\in E(H)$ such that $X\subseteq E_{j}$ \cite{bergc}.

\begin{Lemma}\label{lem2.20}
For a hypergraph $H$ and for any $k\in
\mathbb{Z}^{+}$, a set $S\subset V(H)$ is locating-dominating in $H$
if and only if it is locating-dominating in $H_{k}$. Moreover,
$\lambda(H)=\lambda(H_{k})$.
\end{Lemma}

\begin{proof}
Let $S_{1}$ and $S_{2}$ be two locating-dominating sets for $H$ and
$H_{k}$, respectively. Let $N[u]$ and $N[v]$ be closed neighborhoods
of $u,v\in V(H)$. By the definition of $k$-section hypergraph, we note that
$N[u] = N[v]$ in $H$ if and only if $N[u]=N[v]$ in $H_{k}$. Hence, by
Lemma \ref{lem2.4}, $S_{1}$ and $S_{2}$ are the same sets.
\end{proof}

For any vertex $v\in V(H)$, let $E(v)=\{E_{i}\in E(H)\ |\ v\in E_{i}\}$. Then the
{\it edge degree} of $v$ is $d_{e}(v)=|E(v)|$. If $d_{e}(v)=1$, then we say that $v$ is a {\it pendant
vertex}.

The {\it natural partition} in a hypergraph $H$ is a partition
$P=\{P_{1}, P_{2},\ldots, P_{t}\}$ of $V(H)$ such that for
every pair $u,v\in V(H)$, $u,v\in P_{i}$ if and only if $E(u)=E(v)$. The elements $P_{i}\in P$ are called the {\it levels}
of $H$. A hypergraph $H_{L}=(V(H_L),E(H_L))$ is called the {\it level hypergraph} of $H$ if we
delete every vertex except one from each level of $H$.

\begin{Example}\label{exp2.25}
Let $H$ be the hypergraph with $V(H)=\{v_{1},v_{2},\ldots,v_{10}\}$ and $E(H)=\{E_{1},
E_{2},E_{3}\}$, where $E_{1}=\{v_{1},v_{2},v_{3},v_{4}\}$,
$E_{2}=\{v_{4},v_{5},v_{6}\}$ and
$E_{3}=\{v_{6},v_{7},v_{8},v_{9},v_{10}\}$. Then the collection $P=\{P_{1},P_{2},P_{3},P_{4},P_{5}\}$ with
$P_{1}=\{v_{1},v_{2},v_{3}\}$, $P_{2}=\{v_{4}\}$, $P_{3}=\{v_{5}\}$,
$P_{4}=\{v_{6}\}$, and $P_{5}=\{v_{7},v_{8},v_{9},v_{10}\}$ is
called the natural partition of $H$. The elements $P_{i}\in P$ are
called the levels of $H$. The level hypergraph of $H$ is
the graph with vertex set
$V(H_{L})=\{v_{1},v_{4},v_{5},v_{6},v_{7}\}$ and the edge set
$E(H_{L})=\{E_{1}', E_{2}',E_{3}'\}$ with
$E_{1}'=\{v_{1},v_{4}\}$, $E_{2}'=\{v_{4},v_{5},v_{6}\}$ and
$E_{3}'=\{v_{6},v_{7}\}$.
\end{Example}


\begin{Theorem}\label{th2.26}
If $H_{L}$ is the level hypergraph of a hypertree $H$, then the set
$$ S=\{v\in V(H_L)\ |\ E'(v) = \{E'_i\}\ \mbox{in}\ H_L,\ \mbox{where}\ E'_i\in E(H_L)\}$$ is a
locating-dominating set for $H_{L}$.
\end{Theorem}

\begin{proof}
Let $u,v\in V(H_{L})-S$. Then, by the definitions of the level hypergraph $H_{L}$ and the set $S$,  $u$ and
$v$ are not pendant vertices. Note that, $N(u)\cap S= \{u'\in
\bigcup\limits_{r=1}^{l_{1}}E_{i_{r}}'\ |\ u'\, \mbox{is\, a\,
pendant\, vertex} \}$ and $N(v)\cap S= \{v'\in
\bigcup\limits_{s=1}^{l_{2}} E_{i_{s}}'\ |\ v'\, \mbox {is\, a\,
pendant\, vertex }\}$, where $E_{i_{1}}', E_{i_{2}}',\ldots,
E_{i_{l_{1}}}'$ and $E_{i_{1}}', E_{i_{2}}',\ldots, E_{i_{l_{2}}}'$
with $2\leq l_{1},l_{2}\leq m$ are the hyperedges such that $u\in
\bigcap\limits_{r=1}^{l_{1}} E_{i_{r}}'$ and $v\in
\bigcap\limits_{s=1}^{l_{2}} E_{i_{s}}'$. Since $H_{L}$ is the level
hypergraph, therefore there is at least one hyperedge $E_{i}'\in
E(H_{L})$ such that $u\in E_{i}'$ but $v\not \in E_{i}'$, which
proves the theorem.
\end{proof}

The \emph{primal graph}, $prim(H)$, of a hypergraph $H$ is a graph
with vertex set $V(H)$ and vertices $x$ and $y$ of $prim(H)$ are
adjacent if and only if $x$ and $y$ are contained in a hyperedge.
The {\it middle graph}, $M(H)$, of $H$ is a subgraph of $prim(H)$
formed by deleting all loops and parallel edges. The $dual$ of $H =
(\{v_1,v_2,\ldots, v_m\}, \{E_1,E_2,\ldots,\\ E_k\})$, denoted by
$H^*$, is the hypergraph whose vertices are $\{e_1,e_2,\ldots,
e_k\}$ corresponding to the hyperedges of $H$ and with hyperedges
$V_i = \{e_j\ :\ v_i\in E_j\ \mbox{in}\ H\}$, where $i = 1
,2,\ldots, m$. In other words, the dual $H^*$ swaps the vertices and
hyperedges of $H$.

From the definition of the primal graph of a hypergraph $H$, note that,
closed neighborhood of any vertex in $H$ is same as in $prim(H)$.
Thus, from Lemma \ref{lem2.4}, we have the following straightforward
result:

\begin{Theorem}\label{th2.27}
Let $H$ be a hypergraph and $prim(H)$ be the primal graph of $H$.
Then $\lambda(H) = \lambda(prim(H)) = \lambda(M(H))$.
\end{Theorem}

The primal graph of the dual $H^*$ of a hypergraph $H$ is not a
simple graph. In this case, the middle graph of $H^*$ is a simple
graph. By using the same argument as above, we have the following
result:

\begin{Theorem}\label{th2.28}
Let $H^*$ be the dual of a hypergraph $H$ and $M(H^*)$ be the middle
graph of $H^*$. Then $\lambda(H^*) = \lambda(M(H^*))$.
\end{Theorem}


\begin{Theorem}\label{th2.29}\cite{cac}
Let $G$ be a graph of order $n\geq 2$. Then $\lambda(G)=1$ if and
only if $G \cong P_{2}$.
\end{Theorem}

\begin{Theorem}\label{th2.30}
Let $H$ be a hypergraph. Then $\lambda(H)=1$ if and
only if $H \cong P_{2}$, where $P_2$ is a 2-uniform linear hyperpath with one hyperedge.
\end{Theorem}

\begin{proof}
Suppose that $H\cong P_{2}$. Then, by Theorem \ref{th2.29}, $\lambda(H)=1$.

Conversely, suppose that $\lambda(H)=1$ and $S=\{v\}$ is a locating-dominating set for $H$. Then $H\cong P_{2}$ otherwise, there exist $v_{1}, v_{2}\in V(H)- S$ such that $N(v_{1})\cap S = N(v_{2})\cap S$.
\end{proof}

\begin{Remark}\label{rem2.31}
Given a hypergraph $H$, the hypergraph resulting from deleting the
repeated and non-maximal edges is denoted by $H^{r}_{nm}$. Then for
any hypergraph $H$, $\lambda(H)$=$\lambda (H^{r}_{nm})$ since the
considered graphs are Sperner.
\end{Remark}

From Theorem \ref{th2.30} and Remark \ref{rem2.31}, it follows that for every
hypergraph $H$, $\lambda (H^{r}_{nm})=1$ if and only if $\lambda
(H^{r}_{nm})\cong P_{2}$, where $P_2$ is a 2-uniform linear hyperpath with one hyperedge.


\begin{thebibliography}{9}
\bibitem{achar1}
B. D. Acharya, Domination in hypergraphs II, New directions, {\em Proc. Int. Conf. ICDM, Mysore, India}, (2008), 1-16.
\bibitem{achar2}
B. D. Acharya, Domination in hypergraphs, {\em AKCE J. Combin.}, {\bf 4}(2)(2007), 117-126.
\bibitem{bergc} C. Berg, Hypergraphs: Combinatoric of finite sets,
\emph{Elsevier, North-Holland}, 1989.
\bibitem{behr}
A. Behr and L. Camarinopoulos, On the domination of hypergraphs by their edges, {\em Disc. Math.}, {\bf 187}(1998), 31-38.
\bibitem{ber} N. Bertrand, I. Charon, O. Hudry and A. Lobstein, Identifying and
locating-dominating codes on chains and cycles, \emph{European J.
Combin.,} \textbf{25}(2004), 969-987.
\bibitem{cac} J. Caceres, C. Hernando, M. Mora, I. M. Pelayo and M. L.
Puertas, Locating dominating codes, \emph{Appl. Math. Comput.},
{\bf 220}(2013), 38-45.
\bibitem{char1} I. Charon, O. Hudry and A. Lobstein, Minimizing the size of an identifying or
locating-dominating code in a graph is NP-hard, \emph{Theor. Comput.
Sci.,} \textbf{290}(2003), 2109-2120.
\bibitem {char2} I. Charon, O. Hudry and A. Lobstein, Identifying and
locating-dominating codes: NP-completeness results for directed
graphs, \emph{IEEE Trans. Inform. Theory}, \textbf{48}(2002),
2192-2200.
\bibitem{col} C. J. Colbourn, P. J. Slater and L .K. Stewart, Locating-dominating
sets in series parallel networks, \emph{Congr. Numer.,} \textbf{56}(1987), 135-162.
\bibitem{fin} A. Finbow and B. L. Hartnell, On locating-dominating
sets and well-covered graphs, \emph{Congr. Numer.,}
\textbf{56}(1987), 135-162.
\bibitem{mel} F. Harary and R. A. Melter, On the metric dimention of
a graph, \emph{Ars Combin}., \textbf{2}(1976), 191-195.
\bibitem{hon} I. Honkala, T. Laihonen and S. Ranto, On locating-dominating codes in
binary Hamming spaces, \emph{Disc. Math. Theor. Comput. Sci.,}
\textbf{6}(2004), 265-282.
\bibitem{jose}
B. K. Jose and Z. Tuza, Hypergraph domination and strong independnce, {\em Appl. Anal. Disc. Math.}, {\bf 3}(2009), 347-358.
\bibitem{rall} D. F. Rall and P. J. Slater, On location-domination
numbers for certian classes of graphs, \emph{Congr. Numer.,}
\textbf{45}(1984), 97-106.
\bibitem{schei} E. R. Scheinerman and D. H. Ullman, Fractional graph theory, \emph{John Wiley and Sons}, 2008.
\bibitem{slater1} P. J. Slater, Fault-tolerant locating-dominating sets,
\emph{Disc. Math.}, \textbf{249}(2002), 179-189.
\bibitem{slater2} P. J. Slater, Dominating and reference sets in a graph, \emph{J. Math.
Phys. Sci.,}  \textbf{22}(1988), 445-455.
\bibitem{slater3} P. J. Slater, Dominating and location in acyclic in
graphs, \emph{Networks}, \textbf{17}(1987), 55-64.
\bibitem{slater4} P. J. Slater, Leaves of trees, Proc. 6th
Southeastern Conf. on Combin., Graph theory, and Computing,
\emph{Congr. Numer.}, \textbf{14}(1975), 549-559.
\end{thebibliography}
\end{document}